\documentclass[10pt]{amsart}

 \usepackage{hyperref}

\hypersetup{colorlinks=true, urlcolor=blue, citecolor=blue, linkcolor=blue}

\usepackage{amsmath,amsfonts, latexsym,graphicx, footmisc, amssymb}

\newcommand{\U}{{\mathcal U}}
\newcommand{\0}{{\mathbf 0}}
\newcommand{\C}{{\mathbb C}}
\newcommand{\Z}{{\mathbb Z}}

\newcommand{\Q}{{\mathbb Q}}
\newcommand{\N}{{\mathbb N}}
\newcommand{\cL}{{\mathbb L}}

\newcommand{\W}{{\mathcal W}}

\newcommand{\strat}{{\mathfrak S}}

\newcommand{\hyp}{{\mathbb H}}

\newcommand{\rank}{\mathop{\rm rank}\nolimits}
\newcommand{\arrow}[1]{\stackrel{#1}{\longrightarrow}}
\newcommand{\Adot}{\mathbf A^\bullet}
\newcommand{\Bdot}{\mathbf B^\bullet}
\newcommand{\Cdot}{\mathbf C^\bullet}

\newcommand{\Fdot}{\mathbf F^\bullet}

\newcommand{\Udot}{\mathbf U^\bullet}
\newcommand{\Kdot}{\mathbf K^\bullet}

\newcommand{\p}{{\mathbf p}}
\newcommand{\cc}{{\operatorname{CC}}}

\newcommand{\vdual}{{\mathcal D}}
\newcommand{\call}{{\mathcal L}}

\newcommand{\lotimes}{\ {\overset{L}{\otimes}}\ }
\newcommand{\lboxtimes}{\ {\overset{L}{\boxtimes}\ }}
\newcommand{\piten}{{\pi_1^*\Adot\lotimes\pi_2^*\Bdot}}

\newtheorem{defn0}{Definition}[section]
\newtheorem{prop0}[defn0]{Proposition}
\newtheorem{conj0}[defn0]{Conjecture}
\newtheorem{thm0}[defn0]{Theorem}
\newtheorem{lem0}[defn0]{Lemma}
\newtheorem{corollary0}[defn0]{Corollary}
\newtheorem{example0}[defn0]{Example}
\newtheorem{remark0}[defn0]{Remark}
\newtheorem{question0}[defn0]{Question}

\newenvironment{defn}{\begin{defn0}}{\end{defn0}}
\newenvironment{prop}{\begin{prop0}}{\end{prop0}}

\newenvironment{thm}{\begin{thm0}}{\end{thm0}}

\newenvironment{cor}{\begin{corollary0}}{\end{corollary0}}
\newenvironment{exm}{\begin{example0}\rm}{\end{example0}}
\newenvironment{rem}{\begin{remark0}\rm}{\end{remark0}}

\newcommand{\defref}[1]{Definition~\ref{#1}}
\newcommand{\propref}[1]{Proposition~\ref{#1}}
\newcommand{\thmref}[1]{Theorem~\ref{#1}}

\newcommand{\corref}[1]{Corollary~\ref{#1}}
\newcommand{\exref}[1]{Example~\ref{#1}}
\newcommand{\secref}[1]{Section~\ref{#1}}
\newcommand{\remref}[1]{Remark~\ref{#1}}

\title[Characteristic Cycles and Euler Obstructions]{Characteristic Cycles and the Relative Local Euler Obstruction}

\subjclass[2010]{32B15, 32C35, 32C18, 32B10}
\keywords{characteristic cycle, constructible complexes, local Euler obstruction, relative local Euler obstruction}

\author{David B. Massey}

\date{}

\begin{document}

\begin{abstract} In this paper, we investigate the local Euler obstruction and the relative local Euler obstruction in terms of constructible complexes of sheaves, characteristic cycles, and vanishing cycles. The fundamental tool that we use is the notion of a characteristic complex for an analytic space embedded in affine space.
\end{abstract}

\maketitle

%\newpage

%\tableofcontents

%\newpage

\section{Introduction} 

The local Euler obstruction, defined by MacPherson in \cite{maceuler} in 1974 has been studied by many researchers (see, for instance, \cite{gonzsprin}, \cite{bdk}, \cite{leteissierpolar}, \cite{bss}, and \cite{ebelingzade}) and is, at this point, a standard pice of data associated to a singular point of a complex analytic space. The local Euler obstruction is an obstruction to extending a stratified radial vector field to a non-zero lift in the Nash modification.

The relative local Euler obstruction, defined in \cite{bmps}, is an analog of the local Euler obstruction for a complex analytic function $f:X\rightarrow\C$ at a point $\p$ which is a stratified isolated critical point of $f$. This is again defined in an obstruction-theoretic way; it is an obstruction to extending the conjugate of the stratified gradient vector field of $f$  to a non-zero lift in the Nash modification. This relative concept is beginning to be studied by a number of other researchers; see, for instance, \cite{ebelingzade2}, \cite{dutertregrulha},  \cite{seadetibarverj}, and \cite{amentetal}.

Our own contribution to  \cite{bmps} appeared in the last section of that paper, where we used derived category techniques to extend the definition of the relative local Euler obstruction to functions with arbitrary critical loci, but - in \cite{bmps} - we referred to this generlaized concept as the {\it defect} of $f$. We also gave an algorithm for calculating the defect of $f$  via a process similar to how we defined the L\^e cycles, L\^e numbers, and their generalizations in \cite{lv1}, \cite{lv2}, \cite{lecycles}, and especially in Remark 1.6 of Part IV of \cite{numcontrol}.

\medskip

In this paper, we recall our earlier characterization of the local Euler obstruction in terms of {\it characteristic complexes} and recall and re-derive some standard properties of the local Euler obstruction. We then recall our general definition/characterization of the defect of $f$ - which we now take as the general definition of the relative local Euler obstruction - in terms of vanishing cycles and characteristic complexes. Finally,  in \thmref{thm:releuler}, we prove a number of basic properties for the relative local Euler obstruction and give some examples.

\medskip

We must begin with a section on the basics of characteristic cycles. Throughout this section and much of this paper, we must assume that the reader is familiar with fundamental aspects of the derived category of bounded constructible complexes of sheaves, perverse sheaves, and the nearby and vanishing cycles. Good references for the theory are \cite{kashsch}, \cite{dimcasheaves}, and \cite{schurbook}.

\medskip

We thank J\"org Sch\"urmann for valuable comments on the first version of this paper.

\section[Characteristic Cycles]{Characteristic Cycles}\label{sec:gecc}

A general reference for details of this section is \cite{singenrich}. 

\smallskip

Throughout this paper, we fix a base ring
$R$ that is a regular, Noetherian ring with finite Krull dimension (e.g., $\Z$,
$\Q$, or  $\C$). This implies that every finitely-generated $R$-module has finite projective dimension (in fact, it
implies that the projective dimension of the module is at most $\dim R$). In fact, in later sections, we fix our base as $\Z$.

\smallskip

We let $\U$ be an open neighborhood of the origin of $\C^{n+1}$, and let $X$ be a closed, analytic subset of $\U$. We let $\mathbf z:=(z_0, \dots, z_n)$ be coordinates on $\U$. 

\smallskip

References for much of what we write below are  \cite{stratmorse} and \cite{numcontrol}.

\smallskip

Recall that the {\it complex link}, $\cL_{X, \mathbf p}$, of $X$ at $\mathbf p$ is the Milnor fiber of a generic affine linear form, restricted to $X$, at $\mathbf p$. That is, the complex link is 
$$
\cL_{X, \mathbf p} := B^\circ_\epsilon(\mathbf p)\cap X\cap V(L-b),
$$
where $B^\circ_\epsilon(\mathbf p)$ is an open ball in $\U$ of radius $\epsilon$, where $0<\epsilon\ll 1$, centered at $\mathbf p$, $L$ is a generic affine linear form which is zero at $\mathbf p$, and $b$ is a complex number such that $0<|b|\ll\epsilon$. The homotopy-type of the complex link is an analytic invariant of the germ of $X$ at $\mathbf p$.

\bigskip

Let $\strat$ be a complex analytic Whitney stratification of $X$, with connected strata. Let $\Fdot$ be a bounded complex of sheaves of $R$-modules on $X$, which is constructible with respect to $\strat$. For each $S\in\strat$, we let $d_S:=\dim S$, and let $(\N_{X, S}, \cL_{X, S})$ denote {\it complex Morse data for $S$ in $X$}, consisting of a normal slice and complex link of $S$ in $X$.  Recall that, if $\mathbf p\in S$, then $\cL_{X, S}$ is the complex link of the normal slice to $S$ at $\mathbf p$, i.e., $\cL_{X, S}=\cL_{\N_{X, S}, \mathbf p}$. The homeomorphism-type of the pair $(\N_{X, S}, \cL_{X, S})$ is independent of the choices.
\smallskip

\begin{defn} For each $S\in\strat$ and each integer $k$, the isomorphism-type of the $R$-module
$m_S^k(\Fdot):=\hyp^{k-d_S}(\N_{X, S}, \cL_{X, S}; \Fdot)$ is independent of the choice of $(\N_{X, S}, \cL_{X, S})$; we
refer to
$m_S^k(\Fdot)$ as the {\it degree $k$ Morse module of $S$ with respect to
$\Fdot$}.
\end{defn}

\smallskip

\begin{rem}\label{rem:morsemod} The shift by $d_S$ above is present so that perverse sheaves can have non-zero Morse modules in only degree $0$. 

We also remark that, up to isomorphism, $m_S^k(\Fdot)$ can be obtained in terms of vanishing cycles. To accomplish this, select any point $\p\in S$. Consider an analytic function $\tilde g:(\U^\prime, \p)\rightarrow (\C,0)$  on some open neighborhood of $\p$ in $\U$ such that $d_\p \tilde g$ is a nondegenerate covector (in the sense of \cite{stratmorse}), and such that $\p$ is a (complex) nondegenerate critical point of $\tilde g_{|_{\U^\prime\cap S}}$. Let $g:=\tilde g_{|_{\U^\prime\cap X}}$. Then, $m_S^k(\Fdot)$ is isomorphic to the stalk cohomology $H^k(\phi_g[-1]\Fdot)_{\p}$. 

 Note that, if $\0$ is a point-stratum, then $m^k_\0(\Fdot)\cong H^k(\phi_{\call}[-1]\Fdot)_\0$, where $\call$ is the restriction to $X$ of a generic linear form $\tilde\call$. In particular, if $\0$ is a point-stratum which is not an isolated point of $X$, 
 $$
 m^k_\0(\Z^\bullet_X)\cong H^k(\phi_{\call}[-1]\Z^\bullet_X)_\0\cong \hyp^{k}(\N_{X, \0}, \cL_{X, \0}; \Z^\bullet_X)\cong\widetilde H^{k-1}(\cL_{X, \0}; \Z),
 $$
 where, in the final term on the right, $\widetilde H$ denotes the usual reduced singular cohomology of the complex link of $X$ at $\0$.
 
 \smallskip
 
 Finally, if $\mathbf p$ is in a stratum $S_0$, and $\N_{X, S_0}$ is a normal slice to $S$ in $X$ at $\mathbf p$, then $\N_{X, S_0}$ is stratified by 
 $$\{\N_{X, S_0}\cap S \ | \ S\in \strat\},$$
 though the strata need not be connected. Nonetheless, if $Y:=\N_{X, S_0}$ and $\hat S:=Y\cap S$, it is trivial that, up to homeomorphism, $(\N_{Y, \hat S}, \cL_{Y, \hat S})$ (at any point of $\hat S$) is given by $(\N_{X, S}, \cL_{X, S})$, simply because transverse intersections of transverse intersections are transverse intersections.  
 
 \end{rem}

\bigskip

Now, for any analytic submanifold $M\subseteq\U$, we denote the conormal space $$\{(\p, \omega)\in T^*\U\ |\ \omega(T_{\p}M)\equiv 0\}$$ by $T^*_M\U$, and will typically be interested in its closure $\overline{T^*_M\U}$ in $T^*\U$.

\medskip

\begin{defn}\label{def:charcycle}
Suppose that $R$ is an integral domain.

Define $c_S(\Fdot):=\sum_{k\in\Z} (-1)^{k}\rank(m^k_S(\Fdot))$, and define the {\bf characteristic cycle of $\Fdot$ (in $T^*\U$)} to be the analytic cycle
$$\cc(\Fdot):=\sum_{S\in\strat}c_S(\Fdot)\left[\overline{T^*_S\U}\right].
$$
We write $c_\0(\Fdot)$ in place of $c_{\{\0\}}(\Fdot)$, and let $c_\0(\Fdot)=0$ if $\{\0\}\not\in\strat$.

\medskip

The underlying set $\left|\cc(\Fdot)\right|=\bigcup_{c_S(\Fdot)\neq 0}\overline{T^*_S\U}$ is the {\bf characteristic variety of $\Fdot$ (in $T^*\U$)}.

\medskip

For $0\leq k\leq \dim X$, let 
$$\cc_{\geq k}(\Fdot) \ :=\sum_{S\in\strat, \dim S\geq k}c_S(\Fdot)\left[\overline{T^*_S\U}\right].
$$

\end{defn}

\smallskip

Throughout this paper, whenever we refer to $c_S(\Fdot)$ or $\cc(\Fdot)$, we assume that the base ring is an integral domain, even if we do not explicitly state this.

\smallskip

\begin{rem}  There are various conventions for the signs involved in the characteristic cycle. In fact, our definition above uses a different convention than we used in our earlier works. Our current definition is the most desirable when working with perverse sheaves. In hopes of avoiding confusion with our earlier work, we have changed our notation for the characteristic cycle. Note that, using our current convention, the characteristic cycle is not changed by extending $\Fdot$ by zero to all of $\U$.
\end{rem}

\medskip

We give some basic, easy properties of the characteristic cycle concern how they work with shifting, constant sheaves, distinguished triangles, and the Verdier dual $\vdual\Fdot$. The proofs are all trivial, and we leave them to the reader.

\medskip

\vbox{
\begin{prop}\label{prop:basicccprops}

\ 

\begin{enumerate}
\item $\cc(\Fdot[j])=(-1)^j\cc(\Fdot)$.

\smallskip

\item If $X$ is a pure-dimensional (e.g., connected) complex manifold, then 
$$\cc(\mathbf R^\bullet_X) = (-1)^{\dim X}[T^*_X\U],$$
 i.e., $\cc(\mathbf R^\bullet_X[\dim X]) =[T^*_X\U]$.

\smallskip

\item If $\Adot\rightarrow\Bdot\rightarrow\Cdot\arrow{[1]}\Adot$ is a distinguished triangle in $D^b_c(X)$, then $\cc(\Bdot)=\cc(\Adot)+\cc(\Cdot)$.

\smallskip

\item $\cc(\Fdot) = \cc(\vdual\Fdot)$.
\end{enumerate}
\end{prop}
}

\medskip

For calculating the characteristic cycle of the constant sheaf, the following is very useful:

\begin{cor} Suppose that $Y$ and $Z$ are closed analytic subsets of $X$ such that $X=Y\cup Z$. Then, 
$$\cc(\mathbf R^\bullet_X)=\cc(\mathbf R^\bullet_Y)+\cc(\mathbf R^\bullet_Z)-\cc(\mathbf R^\bullet_{Y\cap Z}).
$$
\end{cor}
\begin{proof} Let $j:Y\hookrightarrow X$, $k:Z\hookrightarrow X$, and $l:Y\cap Z\hookrightarrow X$  denote the respective inclusions. Then, there is a canonical distinguished triangle
$$
\mathbf R^\bullet_X\rightarrow j_*j^*\mathbf R^\bullet_X\oplus k_*k^*\mathbf R^\bullet_X\rightarrow l_*l^*\mathbf R^\bullet_X\arrow{[1]}\mathbf R^\bullet_X.
$$
As the pull-back of the constant sheaf is the constant sheaf, and as the characteristic cycle is unaffected by extensions by zero, the desired conclusion follows immediately from Item 3 of \propref{prop:basicccprops}.
\end{proof}

\medskip

It is also easy to describe how the characteristic cycle of normal slices to strata depend on the original characteristic cycle. We use the set-up and remark at the end of \remref{rem:morsemod}.

\begin{prop}\label{prop:charslice} Suppose that $\mathbf p$ is in a stratum $S_0$ of $\strat$, and $\N_{X, S_0}$ is a normal slice to $S$ in $X$ at $\mathbf p$, then $\N_{X, S_0}$ is stratified by 
 $$\{\N_{X, S_0}\cap S \ | \ S\in \strat\},$$
 though the strata need not be connected.  For $S\in\strat$, let $\hat S:=S\cap \N_{X, S_0}$.
 
 Then, we have an equality of Morse modules $m_{\hat S}^k(\Fdot_{|_{\N_{X, S_0}}}[-\dim S_0])=m_S^{k}(\Fdot)$, which implies that
 $$
\cc\big(\Fdot_{|_{\N_{X, S_0}}}[-\dim S_0]\big) = \sum_{S\in \strat}c_S(\Fdot)\left[\overline{T_{S\cap \N_{X, S_0}}^*\U}\right],
 $$
 (Note that a summand would be $0$ unless $S_0\subseteq\overline{S}$.)
\end{prop}

\begin{proof} Let $Y:=\N_{X, S_0}$. Using the end of \remref{rem:morsemod}, we find
 $$m_{\hat S}^k(\Fdot_{|_Y})=\hyp^{k-\dim \hat S}(\N_{Y, \hat S}, \cL_{Y, \hat S}; \Fdot_{|_Y})=$$
 $$
\hyp^{k-\dim S +\dim S_0}(\N_{X, S}, \cL_{X, S}; \Fdot)=m_S^{k+\dim S_0}(\Fdot),
 $$
 i.e., $m_{\hat S}^k(\Fdot_{|_Y}[-\dim S_0])=m_S^{k}(\Fdot)$. The conclusions follow at once.
\end{proof}

\medskip Note that 
$$
\left[\overline{T_{S\cap \N_{X, S_0}}^*\U}\right] = \sum_{C\subset S\cap \N_{X, S_0}}\left[\overline{T_C^*\U}\right],
$$
where $C$ runs over the connected components of $S\cap \N_{X, S_0}$.

 \bigskip

The following proposition is immediate from formula 5.6 of \cite{schurbook}.

\begin{prop}\label{prop:boxprod}
Suppose that $R$ is a principal ideal domain. Let $X$ and $Y$ be analytic spaces, let $\pi_1:X\times Y\rightarrow X$ and $\pi_2:X\rightarrow Y$ denote the projections. Let $\strat$ and $\strat'$ be Whitney stratifications of $X$ and $Y$, respectively. Let $\Adot$ and $\Bdot$ be bounded,  complexes of sheaves on $X$ and $Y$, respectively, which are constructible with respect to $\strat$ and $\strat'$, respectively. Let $\Adot\lboxtimes\Bdot:=\piten$.

Then, $\Adot\lboxtimes\Bdot$ is constructible with respect to the product stratification 
$$\{S\times S'\ | \ S\in \strat, \ S'\in\strat'\}$$
 and, for all $S\in\strat$ and $S'\in \strat'$,

\medskip

$m_{S\times S'}^k\big(\Adot\lboxtimes\Bdot\big) =\hfill$

\smallskip

$\hfill\displaystyle\bigoplus_{i+j=k}m^i_S(\Adot)\otimes m^j_{S'}(\Bdot)  \ \oplus \ \bigoplus_{i+j=k+1}\operatorname{Tor}\big(m^i_S(\Adot), m^j_{S'}(\Bdot)\big)$.

\medskip

Consequently,
$$
c_{S\times S'}\big(\Adot\lboxtimes\Bdot\big)= c_S(\Adot)\cdot c_{S'}(\Bdot).
$$
\end{prop}

\bigskip

We also have the following simple result, well-known to experts, whose proof we include for completeness. Recall the definition of $\cc_{\geq k}(\Fdot)$ from \defref{def:charcycle}.

\smallskip

\begin{prop}\label{prop:charstrat} Suppose that $\Adot$ and $\Bdot$ are bounded, constructible complexes of sheaves on the $d$-dimensional analytic space $X$. Suppose that $\strat$ is a stratification with respect to which both $\Adot$ and $\Bdot$ are constructible (which always exists). 

Then, $\cc_{\geq k}(\Adot)=\cc_{\geq k}(\Bdot)$  if and only if, for all $S\in \strat$ such that $\dim S\geq k$, for all $\mathbf p\in S$, there is an equality of Euler characteristics of the stalk cohomology $\chi(\Adot)_{\mathbf p}=\chi(\Bdot)_{\mathbf p}$.

In particular, $\cc(\Adot)=\cc(\Bdot)$  if and only if, for all $\mathbf p\in X$, $\chi(\Adot)_{\mathbf p}=\chi(\Bdot)_{\mathbf p}$.
\end{prop}
\begin{proof} The proof is by downward induction on $k$. Certainly the result is trivial for $k=d$. Now suppose that $k_0\geq 0$ and that the statement is true for all $k$ such that $k_0+1\leq k\leq d$; we wish to show that the statement is true for $k=k_0$.

Let $S_0\in\strat$ be a stratum of dimension $k_0$, and let $\mathbf p_0\in S_0$. For each stratum $S$ of dimension greater than or equal to $k_0+1$, let $\p_S$ denote a point of $S$. If we let $\Fdot$ be $\Adot$ or $\Bdot$, then
$$c_{S_0}(\Fdot):=\chi(\N_{S_0}, \cL_{S_0}; \Fdot[-k_0])= \chi(\N_{S_0}; \Fdot[-k_0])-\chi( \cL_{S_0}; \Fdot[-k_0])=$$
$$
(-1)^{k_0}\Big\{\chi(\Fdot)_{\p_0} -\sum_{S, \dim S\geq k_0+1}\chi\big(\cL_{S_0}\cap S\big)\cdot\chi(\Fdot)_{\p_S} \Big\}.
$$
Note that our inductive hypothesis implies that the summation on the right above is the same whether $\Fdot$ equals $\Adot$ or $\Bdot$.

Therefore, $c_{S_0}(\Adot)= c_{S_0}(\Bdot)$ if and only if $\chi(\Adot)_{\p_0}=\chi(\Bdot)_{\p_0}$, and we are finished.
\end{proof}

\medskip

Below, we use $\psi_f$ and $\phi_f$ to denote the nearby and vanishing cycles along $f$, respectively; we also frequently subtract from $f$ the possibly non-zero constant $f(\mathbf p)$ when working at a point $\mathbf p$. Combined with \propref{prop:charstrat}, what we prove below is the well-known fact that the constructible functions given by taking the Euler characteristics of the nearby and vanishing cycles of a complex along a function $f$ depend only on $f$ and the constructible function given by taking the Euler characteristics of the stalks of the complex.

\smallskip

\begin{cor}\label{cor:ccvan} Suppose that $\cc(\Adot)=\cc(\Bdot)$ and that we have a complex analytic $f:X\rightarrow\C$. Then, for all, $\mathbf p\in X$,
$$
\chi(\psi_{f-f(\mathbf p)}\Adot)_{\mathbf p}=\chi(\psi_{f-f(\mathbf p)}\Bdot)_{\mathbf p}\hskip 0.2in\textnormal{and}\hskip 0.2in \chi(\phi_{f-f(\mathbf p)}\Adot)_{\mathbf p}=\chi(\phi_{f-f(\mathbf p)}\Bdot)_{\mathbf p}.
$$
\end{cor}
\begin{proof} For convenience, we shall assume that $f(\p)=0$. Let $F_{f, \p}$ denote the Milnor fiber of $f$ at $\p$. Once again, choose a Whitney stratification $\strat$ with respect to which both $\Adot$ and $\Bdot$ are constructible and, for each $S\in\strat$, select a $\p_S\in S$.

Then,
$$
\chi(\psi_f\Adot)_{\mathbf p}=\chi(F_{f, \p}; \Adot)=\sum_{S\in\strat}\chi(F_{f,\p}\cap S)\cdot\chi(\Adot)_{\p_S}.
$$
By the proposition, this also equals $\chi(\psi_f\Bdot)_{\mathbf p}$.

The result about the vanishing cycles follows immediately since
$$
 \chi(\phi_f\Adot)_{\mathbf p}= \chi(\psi_f\Adot)_{\mathbf p}-\chi(\Adot)_\p = \chi(\psi_f\Bdot)_{\mathbf p}-\chi(\Bdot)_\p= \chi(\phi_f\Bdot)_{\mathbf p}.
$$
\end{proof}

\section{Characteristic Complexes}\label{sec:charcom}

For the remainder of this paper, we fix our base ring to be $\Z$.

\smallskip

Some classical constructions in the study of singular spaces, such as calculating the polar varieties and polar multiplicities of L\^e and Teissier and the Nash modification, deal with contributions from only the smooth strata of $X$. From our point of view, these are results where the underlying complex of sheaves is a {\it characteristic complex}.  

Note that, while our definition of the characteristic cycle in this paper is not what we used in \cite{bmps}, our definition below of a characteristic complex has been adjusted in such a way that the same complexes here and in \cite{bmps} are characteristic complexes.

\begin{defn}\label{def:charcom} Let $X=\bigcup_i X_i$ be the decomposition of $X$ into its irreducible components. 

We say that a complex of sheaves $\Kdot$ on $X$ is a {\bf characteristic complex for $X$} provided that
$$
\cc(\Kdot) = \sum_i (-1)^{\dim X_i}\left[\overline{T^*_{(X_i)_{\operatorname{reg}}}\U}\right].
$$

Thus, $\Kdot$ is a characteristic complex if and only if there exists a complex analytic Whitney stratification $\strat$ of $X$, with connected strata, with respect to which $\Kdot$ is constructible and such that, for all $S\in\strat$, the Euler characteristic $c_S(\Kdot)$ of the Morse modules of  $S$ with respect to $\Kdot$ is zero {\bf unless} $S$ is an open dense subset of one of the $(X_i)_{\operatorname{reg}}$, in which case 
$$c_S(\Kdot)=(-1)^{\dim S}=(-1)^{\dim X_i}.
$$
\end{defn}

\medskip

\begin{rem}\label{rem:vanzero} Note that \remref{rem:morsemod} implies that, if $\mathbf p$ is not an isolated point in $X$ and $\Kdot$ is a characteristic complex for $X$, then for a generic choice of (restricted) linear forms $\call$, $\chi(\phi_{\call-\call(\mathbf p)}\Kdot)_{\mathbf p}=0$, i.e., $\chi(\Kdot)_{\mathbf p}=\chi(\psi_{\call-\call(\mathbf p)}\Kdot)_{\mathbf p}$.
\end{rem}

\medskip

\begin{exm} Suppose that $X$ is a complex manifold with connected components $\{X_i\}_i$. Then it is immediate that $\Z^\bullet_X$ is characteristic complex for $X$.

\end{exm}

\medskip

\begin{prop}\label{prop:ccc} Let $X=\bigcup_i X_i$ be the decomposition of $X$ into its irreducible components.  Suppose that, for each $i$, $\Kdot_i$ is a characteristic complex for $X_i$ and let ${\widehat{\mathbf K}}_i^\bullet$ denote the extension by zero of $\Kdot_i$ to all of $X$. Then, $\bigoplus_i {\widehat{\mathbf K}}_i^\bullet$ is a characteristic complex for $X$.
\end{prop}
\begin{proof} This is immediate from Item 3 of \propref{prop:basicccprops}.
\end{proof}

\medskip

We wish to give a simple, but important, example of characteristic complexes.

\begin{exm}\label{exm:charcom} Suppose that $X$ pure-dimensional, of dimension $>0$, and has a single singular point at the origin in $\U$. Then,
$$
\cc\left(\Z ^\bullet_X\right) = (-1)^{\dim X}\left[\overline{T^*_{X_{\operatorname{reg}}}\U}\right] + b\left[T^*_{\0}\U\right],
$$
where the final part of \remref{rem:morsemod} tells us that $b=1-\chi(\cL_{X, \0})$.

To produce a characteristic complex for $X$, we must modify $\Z ^\bullet_X$ at the origin. This is easy.

If $b=0$, there is nothing to do; $\Z ^\bullet_X$ would be a characteristic complex. 

\smallskip

If $b<0$, let $\Adot$ be the extension by zero to all of $X$ of the constant sheaf $\Z^{-b}$ on the point-set $\{\0\}$. Then, $\Z ^\bullet_X\oplus \Adot$ is a characteristic complex for $X$.

\smallskip

If $b>0$, let $\Adot$ be the extension by zero to all of $X$ of the shifted constant sheaf $\Z^b[1]$ on the point-set $\{\0\}$. Then, $\Z ^\bullet_X\oplus \Adot$ is a characteristic complex for $X$.

\medskip

Note that, in each of these cases, the stalk cohomology at the origin of the resulting characteristic complex is $\chi(\cL_{X, \0})$.
\end{exm}

\bigskip

The following proposition is well-known to experts in the form: $\cc$ yields a surjection from the Grothendieck group of constructible complexes to the group of Lagrangian cycles. For completeness, we give the proof, which is basically induction on the construction given in \exref{exm:charcom}.

\begin{prop} Let $\strat$ be any Whitney stratification of $X$, with connected strata. Then, there exists a characteristic complex on $X$ which is constructible with respect to $\strat$. In particular, characteristic complexes exist for all $X$.
\end{prop}
\begin{proof} This proof is contained in Lemma 3.1 of \cite{hypercohom}. However, we wish to sketch it here.

Note that \propref{prop:ccc} implies that we need deal only with the case where $X$ is irreducible. Hence, we assume that $X$ is irreducible of dimension $d$.

\smallskip

Let $\strat$ be a Whitney stratification of $X$, with connected strata.
\smallskip

For every stratum $S\in \strat$ and every non-negative integer $v$, let $\Udot_{S, v}$ denote the extension by zero to all of $X$ of $\big(\Z^\bullet_S\big)^v$ so that $c_S(\Udot_{S,v})=(-1)^{\dim S}v$ (where $c_S$ is the coefficient of $\big[\overline{T_S^*\U}\big]$ in the characteristic cycle). If $v$ is a negative integer, define $\Udot_{S,v}:=\Udot_{S,-v}[1]$ so that, again,  $c_S(\Udot_{S,v})=(-1)^{\dim S}v$.

Now we construct a characteristic complex as a direct sum, canceling out conormal cycles over lower-dimensional strata.  Let 
$$\Kdot_d=\Kdot_{\geq d} :=\Udot_{X_{\operatorname{reg}}, 1}.
$$
Note that $\Kdot_d$ is also constructible with respect to $\strat$ and, if $S\in \strat$ has dimension $d$, then $c_S(\Kdot_d)=(-1)^d$.

Now we need cancel out the contributions to the characteristic cycle from lower-dimensional strata. 

Let 
$$\Kdot_{d-1}:=\bigoplus_{S\in \strat, \dim S=d-1}\Udot_{S,-c_S(\Kdot_d)},$$
so that $\Kdot_{\geq d-1}:=\Kdot_d\oplus\Kdot_{d-1}$ has the property that, for $S\in \strat$ of dimension at least $d-1$,
$$
c_S\left(\Kdot_{\geq d-1}\right)=
\begin{cases}
(-1)^d, \textnormal{ if } \dim S=d;\\
0, \textnormal{ if }\dim S=d-1.
\end{cases}
$$

Continuing in this manner, we produce $\Kdot:=\Kdot_{\geq 0}$ which is a characteristic complex for $X$.
\end{proof}

\medskip

\begin{prop}\label{prop:kextprod} Let $X$ and $Y$ be analytic spaces, let $\pi_1:X\times Y\rightarrow X$ and $\pi_2:X\rightarrow Y$ denote the projections.  Let $\Kdot_X$ and $\Kdot_Y$ be characteristic complexes for $X$ and $Y$, respectively. Let $\Kdot_X\lboxtimes\Kdot_Y:=\pi_1^*\Kdot_X\lotimes\pi_2^*\Kdot_Y$.

Then, $\Kdot_X\lboxtimes\Kdot_Y$ is a characteristic complex for $X\times Y$.

\end{prop}
\begin{proof} This is immediate from Item 8 of \propref{prop:boxprod}.
\end{proof}

\medskip

\section{The Local Euler Obstruction}\label{sec:leo}

Our primary interest in characteristic complexes lies in their relationship with the local Euler obstruction, defined by MacPherson in \cite{maceuler}. We let $\operatorname{Eu}_{\mathbf p}X$ denote the local Euler obstruction of $X$ at $\mathbf p$, and first recall some well-known results, which appear either explicitly or implicitly in \cite{maceuler}. Note that we slightly extend the usual definition of the local Euler obstruction to possibly non-pure-dimensional spaces by adding over the irreducible components.

\smallskip

\begin{prop}\label{prop:eulerprops}

\ 

\begin{enumerate}
\item The local Euler obstruction is, in fact, local, i.e., if $\W$ is an open neighborhood of $\mathbf p$ in $X$, then $\operatorname{Eu}_{\mathbf p}X=\operatorname{Eu}_{\mathbf p}\W.$
\smallskip
\item If $\mathbf p$ is a smooth point of $X$, then $\operatorname{Eu}_{\mathbf p}X=1$.
\smallskip
\item If $(\mathbf x, \mathbf y)\in X\times Y$, then $\operatorname{Eu}_{(\mathbf x, \mathbf y)}(X\times Y)=\left(\operatorname{Eu}_{\mathbf x}X\right)\left(\operatorname{Eu}_{\mathbf y}Y\right)$.
\smallskip
\item If $\mathbf p\in X$ and $X_i$ denotes the local irreducible components of $X$ at $\mathbf p$, then $\operatorname{Eu}_{\mathbf p}X=\sum_i\operatorname{Eu}_{\mathbf p}X_i.$
\smallskip
\item $\operatorname{Eu}_{\mathbf x}X$ is a constant function of $\mathbf x$ along the strata of any Whitney stratification of $X$ (which has connected strata).
\end{enumerate}
\end{prop}

\medskip

There is also the important result:

\begin{thm}\label{thm:bdk} (Brylinski, Dubson, Kashiwara, \cite{bdk}) Suppose that $\Adot$ on $X$ is constructible with respect a Whitney stratification $\strat$, and that
$$
\cc(\Adot)=\sum_{S\in\strat}c_S(\Adot)\left[\overline{T^*_S\U}\right].
$$
Then, for all $\mathbf p\in X$,
$$
\chi(\Adot)_{\mathbf p}=\sum_{S\in\strat} (-1)^{\dim S}c_S(\Adot)\operatorname{Eu}_{\mathbf p}\overline S,
$$
where we set $\operatorname{Eu}_{\mathbf p}(\overline S)=0$ if $\mathbf p\not\in\overline{S}$.
\end{thm}

\medskip

{\bf From this, we immediately conclude the fundamental relationship between characteristic complexes and the local Euler obstruction}:

\begin{cor}\label{cor:eulerob} Let $\Kdot$ be a characteristic complex for $X$. Let $\mathbf p\in X$. 

Then,
$$
\operatorname{Eu}_{\mathbf p}X=\chi\big(\Kdot\big)_{\mathbf p}.
$$
\end{cor}

\medskip

\begin{rem} We note, as in \cite{schurpaper}, Remark 0.1, that much of what we have written can be described just using the language of constructible functions. Suppose that, for every constructible complex of sheaves $\Fdot$, we let $$\alpha_{{}_{\Fdot}}:X\rightarrow \Z
$$
denote the constructible function $\alpha_{{}_{\Fdot}}(\mathbf x):=\chi(\Fdot)_{\mathbf x}$.

Then, it is well-known that $\alpha_{{}_{()}}$ yields a surjection from the Grothendieck group of constructible complexes on $X$ to the group of constructible functions on $X$.

As $\mathbf p \mapsto\operatorname{Eu}_{\mathbf p}X$ is a constructible function, our definition of a characteristic complex $\Kdot$ was precisely designed so that $\alpha_{{}_{\Kdot}}$ equals the local Euler obstruction function.
\end{rem}

\medskip

As a corollary to \corref{cor:eulerob}, and using the additivity of the Euler characteristic of hypercohomology over complex stratifications, we recover the formula of \cite{bls}:

\begin{cor}\label{cor:indeuler} Let $\strat$ be a complex analytic Whitney stratification, with connected strata, of $X$. For each $S\in\strat$, let $\mathbf p_S$ denote a point in $S$. Let $\mathbf p$ be an arbitrary non-isolated point of $X$ and let $\cL_{X, \mathbf p}$ denote the complex link of $X$ at $\mathbf p$.

Then,
$$
\operatorname{Eu}_{\mathbf p}X=\sum_{S\in\strat}\chi(\cL_{X, \mathbf p}\cap S)\operatorname{Eu}_{\mathbf p_S}X.
$$
\end{cor}
\begin{proof} For convenience, we assume that $\mathbf p=\0$ and that $\call$ is a generic linear form. Let $\Kdot$ be a characteristic complex for $X$. As $\mathbf p$ is not an isolated point of $X$, \remref{rem:vanzero} implies that $\chi(\Kdot)_{\0}=\chi(\psi_{\call}\Kdot)_{\0}$.  

Thus, using the additivity of the Euler characteristic of hypercohomology over complex stratifications, we find
$$
\operatorname{Eu}_{\0}X=  \chi\big(\Kdot\big)_{\0}= \chi(\psi_{\call}\Kdot)_{\0}=\chi\left(\hyp^*(\cL_{X, \0};\,\Kdot)\right) = 
$$
$$
\sum_{S}\chi(\cL_{X, \0}\cap S)\chi(\Kdot)_{\mathbf p_S}=\sum_{S\in\strat}\chi(\cL_{X, \0}\cap S)\operatorname{Eu}_{\mathbf p_S}X.
$$
\end{proof}

\medskip

Before we leave this section, we wish to give a known example of how \corref{cor:indeuler} enables one to calculate local Euler obstructions.

\begin{exm} Suppose that $\mathbf p$ is an isolated singular point of $X$. Then, every point in the complex link, $\cL_{X, \mathbf p}$, is a smooth point of $X$ and, hence, the local Euler obstruction of $X$ at each point of $\cL_{X, \mathbf p}$ is $1$. Thus, we conclude from \corref{cor:indeuler} that
$$
\operatorname{Eu}_\mathbf p X = \chi(\cL_{X, \mathbf p}).
$$

\medskip

In particular, suppose that $X$ is a curve, and $\mathbf p\in X$. Then the complex link $\cL_{X, \mathbf p}$, is a finite collection of points; the number of points is simply the multiplicity, $\operatorname{mult}_\mathbf pX$, of $X$ at $\mathbf p$. Thus, for a curve, we conclude that
$$
\operatorname{Eu}_\mathbf p X =  \chi(\cL_{X, \mathbf p}) =\operatorname{mult}_\mathbf pX.
$$

\medskip

Now suppose that $X$ is 2-dimensional in a neighborhood of $\mathbf p\in X$. Let $S$'s denote  1-dimensional Whitney strata which contain $\mathbf p$ in their closures. Let $\mathbf p_S$ denote an arbitrary point of $S$ near $\mathbf p$. Then, we leave it as an exercise for the reader to use  \corref{cor:indeuler} to conclude that
$$
\operatorname{Eu}_\mathbf p X =  \chi(\cL_{X, \mathbf p})+\sum_S(\operatorname{mult}_\mathbf p\overline S)(\operatorname{mult}_{\mathbf p_S}X-1).
$$
\end{exm}

\medskip

\section{The Relative Local Euler Obstruction}\label{sec:rleo}

 We now wish to discuss the {\it relative local Euler obstruction}, as was introduced  in \cite{bmps}.
 
 \medskip

Recall that $\U$ is an open neighborhood of the origin of $\C^{n+1}$, $X$ is a closed, analytic subset of $\U$. We let $\mathbf z:=(z_0, \dots, z_n)$ be coordinates on $\U$. We identify the cotangent space $T^*\U$ with $\U\times\C^{n+1}$ by mapping $(\p, w_0d_{\p}z_0+\dots+w_nd_{\p}z_n)$ to $(\p, (w_0,\dots, w_n))$. Let $\pi:T^*\U\rightarrow\U$ denote the projection.  

Suppose that we have $\p\in X$ and a complex analytic $f:X\rightarrow\C$. We let $\tilde f$ be a local extension of $f$ at $\p$ to an open neighborhood of $\p$ in $\U$; we assume now that $\U$ is re-chosen to be this (possibly) smaller open neighborhood of $\p$. We also let $d\tilde f$ denote the section of the cotangent bundle to $\U$ given by $d\tilde f(\mathbf x)=(\mathbf x, d_{\mathbf x}\tilde f)$; we let $\operatorname{im}(d\tilde f)$ denote the image of this section in $T^*\U$.

Thus, in coordinates, 
$$\operatorname{im}(d\tilde f) = V\left(w_0-\frac{\partial\tilde f}{\partial z_0}, \dots, w_n-\frac{\partial\tilde f}{\partial z_n} \right).
$$
Note that $\pi$, restricted to $\operatorname{im}(d\tilde f)$, is an isomorphism onto $\U$, with inverse given by $\mathbf x\mapsto (\mathbf x, d_{\mathbf x}\tilde f)$. In particular, we have an isomorphism 
$$\overline{T_{X_{\operatorname{reg}}}^*\U}\cap\operatorname{im}(d\tilde f)\cong \pi\left(\overline{T_{X_{\operatorname{reg}}}^*\U}\cap\operatorname{im}(d\tilde f)\right).
$$
In \cite{critpts}, we gave a name to this last analytic set:

\begin{defn} The {\bf conormal-regular critical locus}, $\Sigma_{\operatorname{cnr}}f$, of $f$ is defined to be
$$
\Sigma_{\operatorname{cnr}}f:= \pi\left(\overline{T_{X_{\operatorname{reg}}}^*\U}\cap\operatorname{im}(d\tilde f)\right) =\Big\{\mathbf p\in \U \ | \ (\mathbf p, d_{\mathbf p}\tilde f)\in \overline{T_{X_{\operatorname{reg}}}^*\U}\Big\}.
$$
\end{defn}

Below, when we take intersection cycles and numbers, we will always be in the case of proper intersections inside the complex analytic manifold $T^*\U$ or inside $\U$ itself. In this case, there are well-defined intersection cycles ({\bf not} cycles classes modulo rational equivalence); see \cite{fulton}.

\bigskip

Assuming that $X$ is pure-dimensional,  the relative local Euler obstruction, $\operatorname{Eu}_{\mathbf p}f$, is defined as an obstruction to extending the conjugate of the stratified gradient vector field of $f$  to a non-zero lift in the Nash modification, provided that $\p$  is a stratified isolated critical point of $f$; see \cite{bmps}.

\medskip

In Corollary 5.4 of \cite{bmps}, we show:

\begin{prop}\label{prop:releuler} Suppose that $X$ is pure-dimensional and that $f:X\rightarrow \C$ has a stratified isolated critical point at $\mathbf p$. Let $\Kdot$ be a characteristic complex for $X$. 
Then, $(\mathbf p, d_{\mathbf p}\tilde f)$ is an isolated point in the intersection $\overline{T_{X_{\operatorname{reg}}}^*\U}\cap\operatorname{im}(d\tilde f)$ (equivalently, $\mathbf p$ is an isolated point in $\Sigma_{\operatorname{cnr}}f$) and

$$
\operatorname{Eu}_{\mathbf p}f =  \chi\big(\phi_{f-f(\mathbf p)}[-1]\Kdot\big)_{\mathbf p}=(-1)^{\dim X}\left(\overline{T^*_{X_{\operatorname{reg}}}\U}\cdot \operatorname{im}(d\tilde f)\right)_{(\mathbf p, d_{\mathbf p}\tilde f)}.
$$
Note that, in the case where $X$ is affine space, this intersection number is the Milnor number of $f-f(\mathbf p)$ at $\mathbf p$.
\end{prop}

\medskip

We can use \propref{prop:releuler} as the basis for generalizing the definition of the relative local Euler obstruction to (possibly) non-isolated critical points of functions on spaces which need not be pure-dimensional. In \cite{bmps}, we referred to this as the {\it defect} $D_{f, X}$.

\bigskip

\begin{defn}\label{def:releuler}  Let $\Kdot$ be a characteristic complex for $X$. Then, we define the {\bf relative local Euler obstruction of $f$} at $\p\in X$ to be
$$
\operatorname{Eu}_{\mathbf p}f := \chi\big(\phi_{f-f(\mathbf p)}[-1]\Kdot\big)_{\mathbf p}.
$$
Note that $\operatorname{Eu}_{\mathbf p}f$ is well-defined by \corref{cor:ccvan}.
\end{defn}

\medskip

\begin{rem}
 Note that the definition immediately implies that, if $c$ is a constant, then $\operatorname{Eu}_{\mathbf p}f=\operatorname{Eu}_{\mathbf p}(f+c)$. Also, note that, in \cite{bmps}, we referred to this generalized relative local Euler obstruction as the {\it defect} $D_{f, X}$.

In terms of constructible functions, what we have done is to define the relative local Euler obstruction to be the constructible shifted vanishing cycle function along $f$ of the local Euler obstruction.
\end{rem}

\bigskip

The relative local Euler obstruction is related to the local Euler obstructions of strata and the Euler characteristics of the intersections of the various strata with the Milnor fiber, $F_{f-f(\mathbf p), \mathbf p}$,  of $f-f(\mathbf p)$ at $\mathbf p$ via the following theorem. We proved this theorem, in slightly different terms, in \cite{bmps}, but the proof is very short, and we include it for completeness.

\medskip 

\begin{thm}\label{thm:bmpsformula} Let $\strat$ be a complex analytic stratification of $X$ such that the local Euler obstruction of $X$ is constant along the strata (e.g., a Whitney stratification with connected strata). For each $S\in\strat$, let $\mathbf p_S$ be a point in $S$. Then, for $\mathbf p\in X$,
$$
\operatorname{Eu}_{\mathbf p}f = \operatorname{Eu}_{\mathbf p}X -\sum_{S\in\strat}\chi(F_{f-f(\mathbf p), \mathbf p}\cap S)\operatorname{Eu}_{\mathbf p_S}X.
$$
\end{thm}

\begin{proof} Let $\Kdot$ be a characteristic complex for $X$. Then,
$$
\operatorname{Eu}_{\mathbf p}f =  \chi\big(\phi_{f-f(\mathbf p)}[-1]\Kdot\big)_{\mathbf p}
$$
which, by the canonical distinguished triangle relating the nearby and vanishing cycles, gives us
$$
\operatorname{Eu}_{\mathbf p}f =   -\chi\big(\Kdot[-1]\big)_{\mathbf p} + \chi\big(\psi_{f-f(\mathbf p)}[-1]\Kdot\big)_{\mathbf p} =
$$
$$
\operatorname{Eu}_{\mathbf p}f  - \chi\big(\hyp^*(F_{f-f(\mathbf p)};\,\Kdot)\big)_\mathbf p=\operatorname{Eu}_{\mathbf p}X -\sum_{S\in\strat}\chi(F_{f-f(\mathbf p), \mathbf p}\cap S)\operatorname{Eu}_{\mathbf p_S}X.
$$

\end{proof}

\medskip

\begin{exm}\label{sabbahex}
 We should note that, given \thmref{thm:bmpsformula}, our definition of the relative local Euler obstruction in terms of vanishing cycles implicitly appears in 4.6 of \cite{sabbahquel}.
\end{exm}

The next two examples contain known results; see, especially, \cite{seadetibarverj}.

\smallskip

\begin{exm}\label{exm:easy} Consider the case where $f:(\U, \0)\rightarrow(\C, 0)$ is a non-locally constant function on an open subset $\U\subseteq \C^{n+1}$. Then, \thmref{thm:bmpsformula} tells us that
$$
\operatorname{Eu}_{\0}f= 1-\chi(F_{f,\0})=-\widetilde\chi(F_{f,\0}),
$$
where $\widetilde\chi$ denotes the Euler characteristic of the reduced cohomology.

In particular, if $\0$ is an {\bf isolated} critical point of $f:(\U, \0)\rightarrow(\C, 0)$, then
$$
\operatorname{Eu}_{\0}f= (-1)^{n+1}\mu_\0(f),
$$
where $\mu_\0(f)$ is the Milnor number of $f$ at $\0$.
\end{exm}

\medskip

\begin{exm}\label{exm:curve}
Now suppose that $\0$ is an isolated singular point of $X$, and we have $f:(X, \0)\rightarrow(\C, 0)$. Then, \thmref{thm:bmpsformula} tells us that
$$
\operatorname{Eu}_{\0}f= \operatorname{Eu}_\0X-\chi(F_{f,\0}) =  \chi(\cL_{X, \0})-\chi(F_{f,\0}) .
$$

In particular, if $f$ itself is a generic linear form, then $\operatorname{Eu}_{\0}f=0$. We shall see in Item 2 below that is true for arbitrarily singular spaces.

As another particular case, suppose that $X$ is a curve and that $\tilde f$ is a local extension of $f$ to the ambient affine space such that $\dim_\0 X\cap V(\tilde f)=0$. Then, we conclude that
$$
\operatorname{Eu}_{\0}f= \operatorname{mult}_\0X-(X\cdot V(\tilde f))_\0.
$$
\end{exm}

\bigskip

We now give a number of basic properties of the relative local Euler obstruction.

\smallskip

\begin{thm}\label{thm:releuler}Let $X=\bigcup_i X_i$ be the decomposition of $X$ into its irreducible components, and let $f_i$ denote the restriction of $f$ to $X_i$.

The relative local Euler obstruction has the following properties:
\begin{enumerate}
\item If $f$ is constant in a neighborhood of $\mathbf p$, then $\operatorname{Eu}_{\mathbf p}f=\operatorname{Eu}_{\mathbf p}X$.
\medskip
\item If $\mathbf p\not\in\Sigma_{\operatorname{cnr}}f$, then $\operatorname{Eu}_{\mathbf p}f =0$. In particular, if $\mathbf p$ is not an isolated point of $X$, and $L$ is the restriction of a generic linear form to $X$, then $\operatorname{Eu}_{\mathbf p}L =0$. 

\medskip
\item $\operatorname{Eu}_{\mathbf p}f =\sum_i\operatorname{Eu}_{\mathbf p}f_i$, where we set $\operatorname{Eu}_{\mathbf p}f_i=0$ if $\p\not\in X_i$.
\medskip
\item If $\p$ is an isolated point in $\Sigma_{\operatorname{cnr}}f$, then 
$$
\operatorname{Eu}_{\mathbf p}f= \sum_i(-1)^{\dim X_i}\left(\overline{T^*_{(X_i)_{\operatorname{reg}}}\U}\cdot \operatorname{im}(d\tilde f)\right)_{(\mathbf p, d_{\mathbf p}\tilde f)}.
$$

\medskip
\item Let $\mathbf q\in Y$ and suppose that we have a complex analytic function $g:Y\rightarrow\C$. Let $f\boxplus g$ denote the function from $X\times Y$ to $\C$ given by $$(f\boxplus g)(x,y)=f(x)+g(y).$$ Then,
$$\operatorname{Eu}_{(\p, \mathbf q)}(f\boxplus g)=\operatorname{Eu}_{\mathbf p}f\cdot \operatorname{Eu}_{\mathbf q}g.$$
\medskip
\end{enumerate}
\end{thm}
\begin{proof} Item 1 follows at once from \thmref{thm:bmpsformula} or, alternatively from \corref{cor:eulerob} and \defref{def:releuler}.

\medskip

Both Items 2 and 4 essentially follow from the vanishing cycle index theorem of Ginsburg \cite{ginsburg},  
and L\^e \cite{leconcept}, but those papers require stronger hypotheses. However, the full results  appears in 4.5 and 4.6 of Sabbah \cite{sabbahquel}, and also in Corollary 0. 3 of Sch\"urmann \cite{schurpaper}.

We can also conclude the results from looking ahead to  \thmref{thm:calceuler} in \secref{sec:calc}.

\medskip

\noindent Item 3:

\smallskip

Suppose that, for each $i$, $\Kdot_i$ is a characteristic complex for $X_i$ and let ${\widehat{\mathbf K}}_i^\bullet$ denote the extension by zero of $\Kdot_i$ to all of $X$. Then, by \propref{prop:ccc}, $\bigoplus_i {\widehat{\mathbf K}}_i^\bullet$ is a characteristic complex for $X$. Thus,
$$
\operatorname{Eu}_{\mathbf p}f=\chi\Big(\phi_{f-f(\p)}[-1]\bigoplus_i {\widehat{\mathbf K}}_i^\bullet\Big)_\p = \sum_i\chi\big(\phi_{f-f(\p)}[-1] {\widehat{\mathbf K}}_i^\bullet\big)_\p=
$$
$$
 \sum_i\chi\big(\phi_{f_i-f_i(\p)}[-1] {\mathbf K}_i^\bullet\big)_\p=\sum_i \operatorname{Eu}_{\mathbf p}f_i.
$$

\medskip

\vbox{\noindent Item 5:

\smallskip

Let us assume, without loss of generality, that $f(\p)=0$ and $g(\mathbf q)=0$. Let $\Kdot_X$ and $\Kdot_Y$ be characteristic complexes for $X$ and $Y$, Then, we know from \propref{prop:kextprod} that $\Kdot_X\lboxtimes\Kdot_Y$ is a characteristic complex for $X\times Y$.

Now, the derived category version of the Sebastiani-Thom Theorem which we proved in \cite{masseysebthom} (but, here, using the more-usual definition/shift on the vanishing cycles) tells us that
$$
H^k\big(\phi_{f\boxplus g}[-1](\Kdot_X\lboxtimes\Kdot_Y)\big)_{(\p, \mathbf q)}\cong H^k\big(\phi_f[-1]\Kdot_X\lboxtimes\phi_g[-1]\Kdot_Y\big)_{(\p, \mathbf q)}\cong
$$

\noindent $\bigoplus_{i+j=k}H^i(\phi_f[-1]\Kdot_X)_\p\otimes H^j(\phi_g[-1]\Kdot_Y)_{\mathbf q} \ \oplus$\hfill

\smallskip

\hfill $\bigoplus_{i+j=k+1}\operatorname{Tor}\big(H^i(\phi_f[-1]\Kdot_X)_\p, H^j(\phi_g[-1]\Kdot_Y)_{\mathbf q}\big)$.

\medskip

\noindent Item 5 follows.}
\end{proof}

\medskip

\begin{rem} We naturally refer to Item 5 above as the {\it Sebastiani-Thom property} of the relative local Euler obstruction.

We should also remark that the quantity 
$$\left(\overline{T^*_{(X_i)_{\operatorname{reg}}}\U}\cdot \operatorname{im}(d\tilde f)\right)_{(\mathbf p, d_{\mathbf p}\tilde f)}$$
 in Item 5 above can be characterized more geometrically as the number of
non-degenerate critical points of a small perturbation of $f$ by a generic linear form $L$ which occur near the origin on $(X_i)_{\operatorname{reg}}$; see Theorem 3.2 of \cite{hypercohom} and Theorem 1.1 of \cite{micromorse}. Thus, Item 5 is a significant generalization of Proposition 2.3 of \cite{seadetibarverj}, as we do not require that $f$ has a stratified isolated critical point.
\end{rem}

\medskip

\begin{exm} Let $(x, y, z)$ be coordinates on $\U:=\C^3$, and let $(w_0, w_1, w_2)$ be the corresponding cotangent coordinates. Let 
$$X:=V(xy)=V(x)\cup V(y)\subseteq\C^3.$$
Let $\tilde f:\U\rightarrow \C$ be given by $\tilde f(x,y,z) := x+y^2+yz$, and let $f$ be the restriction of $\tilde f$ to $X$.

Note that $f$ does {\bf not} possess a stratified isolated singularity at $\0$, since $\Sigma X=V(x,y)=\{\0\}\times \C$ and $f_{|_{V(x,y)}}\equiv 0$. However, we claim that $(\0, d_\0\tilde f)$ {\bf is} an isolated point in $\overline{T^*_{X_{\operatorname{reg}}}\U}\cap\operatorname{im}d\tilde f$.

We find 
$$
\overline{T^*_{X_{\operatorname{reg}}}\U} = V(x, w_1, w_2) \cup V(y, w_0, w_2)
$$
and
$$
\operatorname{im}d\tilde f = V(w_0-1, w_1-2y-z, w_2-y).
$$

Therefore, $\operatorname{im}d\tilde f$ does not intersect $V(y, w_0, w_2)$, and intersects $V(x, w_1, w_2)$ in the single point $V(x, y, z, w_0-1, w_1, w_2)$.

Hence, Item 4 of \thmref{thm:releuler} tells us that
$$
\operatorname{Eu}_{\0}f= (-1)^2\left(\overline{T^*_{(X)_{\operatorname{reg}}}\U}\cdot \operatorname{im}(d\tilde f)\right)_{(\0, d_{\0}\tilde f)}=
$$
$$
\Big(\big(V(x, w_1, w_2) + V(y, w_0, w_2)\big) \cdot  V(w_0-1, w_1-2y-z, w_2-y)\Big)_{(\0, 1,0,0)} = 1.
$$

\end{exm}

\medskip

\begin{exm} Suppose that $Y$ is a curve through the origin in $\C^2$. Let $g:\C^2\rightarrow C$ be such that $\dim_\0 Y\cap V(g)=0$. Let $X:=Y\times \C$, where we use $z$ as the coordinate on this new copy of $\C$. 

Consider the function $f:(X, \0)\rightarrow (\C, 0)$ given by $f(x,y,z)= g(x,y)+z^b$, for some positive integer $b$.

Then, by Item 5 of \thmref{thm:releuler} - the Sebastiani-Thom property - combined with \exref{exm:easy} and the last part of \exref{exm:curve} - we find that
$$
\operatorname{Eu}_{\0}f = \big(\operatorname{mult}_\0Y-(Y\cdot V(g))_\0\big)(1-b).
$$
\end{exm}
 
 \medskip

\section{Calculating the relative local Euler obstruction}\label{sec:calc}

Here, we recall the algorithm and result which is described in Section 6 of \cite{bmps}.

\smallskip

Once again, let $X=\bigcup_i X_i$ be the decomposition of $X$ into its irreducible components, and let $f_i$ denote the restriction of $f$ to $X_i$. Recall from Item 4 of \thmref{thm:releuler} that $\operatorname{Eu}_{\mathbf p}f =\sum_i\operatorname{Eu}_{\mathbf p}f_i$. Therefore, calculating  $\operatorname{Eu}_{\mathbf p}f$ boils down to needing to calculate in the case where $X$ is irreducible.

\medskip

{\bf Thus, suppose throughout this section that $X$ is irreducible, embedded in $\U$, an open subset of $\C^{n+1}$.}

\medskip

\noindent Recall our previous set-up:

\smallskip

We let $\mathbf z:=(z_0, \dots, z_n)$ be coordinates on $\U$, we identify the cotangent space $T^*\U$ with $\U\times\C^{n+1}$ by mapping $(\p, w_0d_{\p}z_0+\dots+w_nd_{\p}z_n)$ to $(\p, (w_0,\dots, w_n))$, and let $\pi:T^*\U\rightarrow\U$ denote the projection.  

Assume that $\p\in X$, that $\tilde f:\U\rightarrow \C$ is a complex analytic function, and that $f$ is the restriction of $\tilde f$ to $X$. We  let $d\tilde f$ denote the section of the cotangent bundle to $\U$ given by $d\tilde f(\mathbf x)=(\mathbf x, d_{\mathbf x}\tilde f)$; we let $\operatorname{im}(d\tilde f)$ denote the image of this section in $T^*\U$.

Thus, in coordinates, 
$$\operatorname{im}(d\tilde f) = V\left(w_0-\frac{\partial\tilde f}{\partial z_0}, \dots, w_n-\frac{\partial\tilde f}{\partial z_n} \right).
$$

\smallskip

Our method of calculation requires a ``generic'' choice  
of coordinates for the ambient, affine space $\U$. This choice of  
coordinates is made as follows. Refine a Whitney stratification of $X$ to  
a stratification $\W:=\{W_j\}$ such that $\W$ satisfies Thom's  
$a_f$ condition and such that $f^{-1}(0)$ is a union of strata of $\W$ (we are  
not assuming that $\W$ is still a Whitney  
stratification). 

Choose the first coordinate $z_0$ so that the hyperplane $z_0^{-1}(0)$  
transversely intersects, in some neighborhood of the origin, all  
positive-dimensional strata of $\{W_j\}$. Then, there is an induced  
stratification (on the germ at the origin) of $X\cap z_0^{-1}(0)$ given  
by $\{W_j\cap z_0^{-1}(0)\}$. We choose $z_1$ so that $z_1^{-1}(0)$  
transversely intersects, in some neighborhood of the origin, all  
positive-dimensional strata of $\{W_j\cap z_0^{-1}(0)\}$. We continue  
in this inductive manner to produce $z_0, \dots, z_{n+1}$. We call such  
a coordinate choice {\it prepolar} (at the origin).

Prepolar coordinates are not as generic as possible, but they are  
generic enough for our purposes. Being prepolar at the origin implies  
that the coordinates are also prepolar at each point in a neighborhood  
of the origin, and we assume that we are in such a neighborhood  
throughout the remainder of this section. 

Assuming that the coordinates are prepolar for $f$ at the origin, all of the intersections that we write below are proper in $T^*\U$ (resp., in $\U$) in a neighborhood of $(\p, d_\p\tilde f)$ (resp., $\p$).

\smallskip

\vbox{\noindent The algorithm is as follows:

\bigskip

The cycle $\Big[\overline{T^*_{X_{\operatorname{reg}}}\U}\Big]$ can be  
written as a sum of purely $(n+1)$-dimensional cycles  
$$\Big[\overline{T^*_{X_{\operatorname{reg}}}\U}\Big]=\widehat\Gamma^{n+1}_{f, \mathbf z}+ 
\widehat\Lambda^{n+1}_{f, \mathbf z},$$ where no component of $\widehat\Gamma^{n+1}_{f, \mathbf z}$ is  
contained in ${\operatorname{im}}(d\tilde f)$ (i.e. in $\overline{T^*_{X_{\operatorname{reg}}}\U}\cap {\operatorname{im}}(d\tilde f)$) in and every component of  
$\widehat\Lambda^{n+1}_{f, \mathbf z}$ is contained in  ${\operatorname{im}}(d\tilde f)$.
}

Now, we define $\widehat\Gamma^k_{f, \mathbf z}$ and $\widehat\Lambda^k_{f, \mathbf z}$ by  
downward induction. If we have defined the purely $(k+1)$-dimensional  
cycle $\widehat\Gamma^{k+1}_{f, \mathbf z}$, then the hypersurface  
$$\Big[V\Big(w_{k+1}-\frac{\partial \tilde f}{\partial  
z_{k+1}}\Big)\Big]
$$ 
properly intersects $\widehat\Gamma^{k+1}_{f, \mathbf z}$ inside  
$U$, and therefore there is a well-defined, purely $k$-dimensional  
intersection cycle
$$
\widehat\Gamma^{k+1}_{f, \mathbf z} \ \cdot\ \Big[V\Big(w_{k+1}-\frac{\partial  
\tilde f}{\partial z_{k+1}}\Big)\Big],
$$
which we can decompose as
$$\widehat\Gamma^{k+1}_{f, \mathbf z} \ \cdot\ \Big[V\Big(w_{k+1}-\frac{\partial \tilde f}{\partial  
z_{k+1}}\Big)\Big]\ =:\ \widehat\Gamma^k_{f, \mathbf z}+\widehat\Lambda^k_{f, \mathbf z},$$
where no component of $\widehat\Gamma^k_{f, \mathbf z}$ is contained in  
${\operatorname{im}}(d\tilde f)$ and every component of  
$\widehat\Lambda^k_{f, \mathbf z}$ is contained in ${\operatorname{im}}(d\tilde f)$.

As the $\widehat\Lambda^k_{f, \mathbf z}$ are contained in  
${\operatorname{im}}(d\tilde f)$, the projection, $\pi$, maps each  
$\widehat\Lambda^k_{f, \mathbf z}$ isomorphically onto a cycle in $\U$; we let  
$\Lambda^k_{f, \mathbf z}:=\pi_*\big(\widehat\Lambda^k_{f, \mathbf z}\big)$ (this is the proper  
projection of a cycle). We refer  
to $\Lambda^k_{f, \mathbf z}$ as the {\it $k$-dimensional L\^e-Vogel cycle}.

Note that in the case where $X=\U$,  
$\Big[\overline{T^*_{X_{\operatorname{reg}}}\U}\Big]=\U\times \{0\}$, and  
the L\^e-Vogel cycles coincide with the L\^e cycles of \cite{lv1}, \cite{lv2}, and \cite{lecycles}.

\vskip .2in

Now,  $\Lambda^k_{f, \mathbf z}$ properly intersects the linear  
subspace $V(z_0-p_0, \dots, z_{k-1}-p_{k-1})$ at $\p$, and we define the {\it  
$k$-dimensional L\^e-Vogel number of $f$ at $\p$} to be the intersection  
number
$$
\lambda^k_{f, \mathbf z}(\p):= \big(\Lambda^k_{f, \mathbf z}\ \cdot\ V(z_0-p_0, \dots, z_{k-1}-p_{k-1})\big)_\p.
$$

\medskip

\begin{thm}\label{thm:calceuler}\textnormal{(Theorem 6.2 of \cite{bmps})} Suppose that $X$ is irreducible, and let $\Kdot$ be a characteristic complex for $X$.

Let $s:=\dim_\p\Sigma_{\operatorname{cnr}}f$ (where we set $s=-\infty$ if $\p\not\in \Sigma_{\operatorname{cnr}}f$), 
 and assume that the coordinates $\mathbf z$ are prepolar for $f$ at $\0$.

\bigskip
Then, the L\^e-Vogel numbers  
$\lambda^k_{f, \mathbf z}(\p)$ are zero if $k>s$, and $$
\operatorname{Eu}_\p f = \chi(\phi_f[-1]\Kdot)_\p = (-1)^{\dim X}\sum_{k=0}^{s}  
(-1)^k\lambda^k_{f, \mathbf z}(\p).
$$

In particular, when $s=0$, the only L\^e-Vogel number which is possibly  
non-zero is $\lambda^0_{f, \mathbf z}(\p)$, and $\lambda^0_{f, \mathbf z}(\p)=  
\big(\overline{T^*_{X_{\operatorname{reg}}}\U}\ \cdot\  
\operatorname{im}(d\tilde f)\big)_{(\p, d_\p\tilde f)}$; thus, when $s=0$,
 $$
\operatorname{Eu}_\p f = (-1)^{\dim X}\big(\overline{T^*_{X_{\operatorname{reg}}}\U}\ \cdot\  
\operatorname{im}(d\tilde f)\big)_{(\p, d_\p\tilde f)}.
$$
\end{thm}

\bigskip

The serious weakness of \thmref{thm:calceuler} is that there is no effective way to obtain the $a_f$ stratification of $X$ which we need in order to know if our coordinates are prepolar or not. The case where $\p$ is an isolated point of $\Sigma_{\operatorname{cnr}}f$ is nice because the resulting formula is independent of the coordinates. However, we can also handle special $1$-dimensional cases fairly easily.

\medskip

\begin{exm} Suppose $X$ is irreducible at $\p$, of dimension $d$, and suppose that $\dim_\p\Sigma X\leq 1$. Then, near $\p$, there is a Whitney stratification of $X$ consisting of the connected components  of $X_{\operatorname{reg}}$, the connected components of $\Sigma X-\{\p\}$, and $\{\p\}$. 

Now, suppose also that $\dim_\p\overline{\Sigma(f_{|_{X_{\operatorname{reg}}}})}\leq 1$. Let $\Sigma:=\Sigma X\cup {\Sigma(f_{|_{X_{\operatorname{reg}}}})}$. Then there is an $a_f$ stratification consisting of the connected components  of $X_{\operatorname{reg}}-V(f)$, the connected components  of $V(f)-\Sigma$, the connected components of $\Sigma-\{\p\}$, and $\{\p\}$. 

Then, one easily sees that the requirement that $(z_0, z_1, \dots, z_n)$ be polar is equivalent to requiring that, near $\p$,

\begin{itemize}

\item $V(z_0-p_0)$ contains no irreducible component of $\Sigma$;

\item for $0\leq i\leq d-1$, $V(z_0-p_0, \dots, z_i-p_i)$ transversely intersects $X_{\operatorname{reg}}$ near $\p$; and

\item for $0\leq i\leq d-1$, $\Sigma \big(f_{|_{X_{\operatorname{reg}}\cap V(z_0-p_0, \dots, z_i-p_i)}}\big)=\emptyset$.

 \end{itemize}

\medskip

\noindent Let us consider a specific example. 

\medskip

Let $\U:=\C^3$ and let $X:=V(y^2-x^3)\subseteq\U$, where we use coordinates $(t,x,y)$ on $\U$. Thus, $X$ is a cross-product of a cusp and $\C$, and has a Whitney stratification consisting of $\{X-V(x,y), V(x,y)\}$. 

\medskip

We wish to determine $\overline{T^*_{X_{\operatorname{reg}}}\U}$. Using $(t, x, y, w_0, w_1, w_2)$ for coordinates on $T^*\U\cong\U\times\C^3$, one looks at the vanishing of $y^2-x^3$ and the $2\times 2$ minors of the matrix 
$$
\left(\begin{matrix}
0& -3x^2& 2y\\
w_0&\phantom{-}w_1&w_2
\end{matrix}\right),
$$
and disposes of those irreducible components which lie over $\Sigma X = V(x,y)$.

We find that, as analytic sets,
$$
 \overline{T^*_{X_{\operatorname{reg}}}\U} = \overline{V(y^2-x^3, w_0x^2, w_0y, 2w_1y+3w_2x^2)- V(x,y)}=
$$
$$
\overline{V(y^2-x^3, w_0, 2w_1y+3w_2x^2)-V(x,y)}.
$$
Thus, as cycles,
\begin{equation}
V(y^2-x^3, w_0, 2w_1y+3w_2x^2)=  \overline{T^*_{X_{\operatorname{reg}}}\U} \ +  \ mV(x,y, w_0),\tag{$\dagger$}
\end{equation}
for some positive integer $m$.

In fact, it is easy to show that
$$
\overline{T^*_{X_{\operatorname{reg}}}\U} = V(y^2-x^3, w_0, 2w_1y+3w_2x^2, 4w_1^2-9w_2^2x),
$$
but, as this is not defined by a regular sequence, it is somewhat problematic to deal with this in the intersections below, so $(\dagger)$ is more useful.

\smallskip

Let $\tilde f:\U\rightarrow\C$ be given by $\tilde f(t,x,y)= 2y-3tx+t^3$, and let $f:={\tilde f}_{|_X}$. Then one easily checks that 
$$
\overline{\Sigma(f_{|_{X_{\operatorname{reg}}}})} = V(x-t^2, y-t^3),
$$
and $(t,x,y)$ are prepolar coordinates at $\0$. Thus, we may use the method of \thmref{thm:calceuler} to calculate the relative local Euler obstruction.

\smallskip

\noindent We will suppress the reference to the coordinate system $(t,x,y)$ in the subscripts below.

\smallskip

We have
$$
\operatorname{im}(d\tilde f)=V\Big(w_0+3(x-t^2), \,w_1+3t, \,w_2-2\Big), 
$$
and one easily finds that
$$
\overline{T^*_{X_{\operatorname{reg}}}\U}\cap \operatorname{im}(d\tilde f) = V\left(y-t^3, x-t^2, w_0, w_1+3t, w_2-2\right),
$$
which is $1$-dimensional. Note for later that this implies that $\widehat\Lambda^2_{f}=0$.

\smallskip

We wish to proceed with the algorithm described just before \thmref{thm:calceuler}, but we get slightly ``tricky'' in the first intersection, to avoid the problem that $\overline{T^*_{X_{\operatorname{reg}}}\U}$ is not defined by a regular sequence.

\smallskip

We begin:

$$\widehat\Gamma^3_{f} = \overline{T^*_{X_{\operatorname{reg}}}\U} .
$$

Now, we use ($\dagger$):
	
$$
V(y^2-x^3, w_0, 2w_1y+3w_2x^2)\cdot V(w_2-2)=  
$$

$$
\overline{T^*_{X_{\operatorname{reg}}}\U}\cdot V(w_2-2) \ +  \ mV(x,y, w_0)\cdot V(w_2-2).
$$

Thus, we conclude that
\begin{equation}
V(y^2-x^3, w_0, w_1y+3x^2, w_2-2)=\widehat\Gamma^2_{f} + mV(x,y, w_0, w_2-2),\tag{$\ddagger$}
\end{equation}
where we have used our earlier observation that $\widehat\Lambda^2_{f}$ must equal $0$.

We wish to investigate the purely $2$-dimensional cycle 
$$Y:=V(y^2-x^3, w_0, w_1y+3x^2, w_2-2).
$$

As sets, $V(w_1)\cap Y=V(x,y, w_0, w_1, w_2-2)$, which is $1$-dimensional. Therefore, we may calculate the cycle structure of $Y$ by looking at the structure where $w_1\neq 0$.

\smallskip

Via this approach, we find
$$
V(y^2-x^3, w_0, w_1y+3x^2, w_2-2) = 
$$
$$
V(27y+w_1^3, 9x-w_1^2, w_0, w_2-2)+3V(x,y,w_0, w_2-2).
$$

Now ($\ddagger$) implies that
$$
\widehat\Gamma^2_{f} = V(27y+w_1^3, 9x-w_1^2, w_0, w_2-2).
$$

We proceed:
$$
\widehat\Gamma^2_{f} \cdot V(w_1+3t) = V(27y+w_1^3, 9x-w_1^2, w_0, w_2-2) \cdot V(w_1+3t)  =
$$

$$
V\left(y-t^3, x-t^2, w_0, w_1+3t, w_2-2\right),
$$
which as a set equals $\overline{T^*_{X_{\operatorname{reg}}}\U}\cap \operatorname{im}(d\tilde f)$.

Therefore, 
$$\widehat\Gamma^1_{f}=0 \textnormal{ and }\widehat\Lambda^1_{f} =V\left(y-t^3, x-t^2, w_0, w_1+3t, w_2-2\right).$$

Since $\widehat\Gamma^1_{f}=0$, we have that $\widehat\Lambda^1_{f} =0$. Thus there is only one non-zero L\^e-Vogel cycle:
$$
\Lambda^1_{f} =\pi_*\big(\widehat\Lambda^1_{f}\big) =\pi_*\big(V(y-t^3, x-t^2, w_0, w_1+3t, w_2-2)\big)= V(y-t^3, x-t^2),
$$
with corresponding L\^e-Vogel number at the origin:
$$
\lambda^1_f(\0)= \big(V(y-t^3, x-t^2)\cdot V(t)\big)_\0=1.
$$

Finally, \thmref{thm:calceuler} tells us that 
$$\operatorname{Eu}_\0f=(-1)^2(-1)\lambda^1_f(\0) = -1.$$

\vskip 0.2in

With a bit of work, one can use \thmref{thm:bmpsformula} to verify this calculation; we leave this as an exercise for the reader.

\end{exm}

\bigskip

\bibliographystyle{plain}
\bibliography{Masseybib}
%\printindex
\end{document}